\documentclass[12pt]{amsart}
\pdfoutput=1
\usepackage{amsmath, amstext, amsbsy, amssymb}
\usepackage{latexsym,array,delarray,epsfig, color,amsfonts,youngtab,yfonts,inslrmin, textcomp, txfonts, pxfonts}
\usepackage{graphicx, epsfig, textcomp, txfonts, pxfonts}
\usepackage{amsmath,amsthm,amsfonts,amssymb,amsxtra, mathtools} 
\usepackage{comment, cancel}
\usepackage{tikz}

\newtheorem{theorem}{Theorem}[section]

\newtheorem{proposition}{Proposition}[section]

\theoremstyle{definition}

\theoremstyle{remark}

\numberwithin{equation}{section}

\newcommand{\al}{\alpha}

\newcommand{\si}{\sigma}

\newcommand{\Z}{\mathbb Z}
\newcommand{\C}{\mathbb C}
\newcommand{\g}{\mathfrak g}
\newcommand{\h}{\mathfrak h}
\newcommand{\cA}{\mathcal A}
\newcommand{\cK}{\mathcal K}

\newcommand{\A}{A_{2n-1}}
\newcommand{\D}{D_{n+1}}
\newcommand{\At}{A_{2n}}
\newcommand{\Df}{D_4}

\newcommand{\dzmw}{\delta (z - w)}

\def\dsum{\displaystyle\sum}

\def\a{\alpha}

\begin{document}

\title[Bosonic free field realization of twisted 2-toroidal Lie algebras]
{Bosonic free field realization of twisted 2-toroidal Lie algebras}
\author{Chad R. Mangum}
\address{Department of Mathematics,
   Niagara University,
   Lewiston, NY 14109, USA}
\email{cmangum@niagara.edu}

\begin{abstract}
Using free fields, we construct a bosonic realization of toroidal Lie algebras of type $\A, \At, \D, \Df$ which are twisted by a Dynkin diagram automorphism. This realization is based upon the recently found Moody-Rao-Yokonuma-like presentation. In particular, our construction contains the first bosonic realization of the twisted affine algebra $\Df^{(3)}$.
\end{abstract}

\maketitle

\section{Introduction} \label{intro}
Since their introduction in the late 1960s, affine (Kac-Moody) algebras \cite{K1, C} have played a significant role in various areas of mathematics and physics, and their representation theory is now well-developed. Many generalizations of affine algebras have been studied, notably double affine Lie algebras \cite{CL}, extended affine Lie algebras \cite{ABGP, NSW} and ($\ell$-)toroidal (Lie) algebras, $\ell \in \Z_{> 0}$.

The $\ell$-toroidal algebras can be realized as the universal central extension of a multi-loop algebra in $\ell$ variables with infinite dimensional center \cite{BK} for $\ell > 1$ (the case $\ell = 1$ gives precisely the affine algebras). Like affine algebras \cite{F3}, toroidal algebras have found applications in physics \cite{B} and also come in both untwisted and twisted types. Much progress has been made on the representation theory in the untwisted case, in particular through the use of vertex operators \cite{BB, BBS, FJW, T1, T2, FM} (see \cite{FLM} for a thorough study of vertex operator algebras). An expository survey can be found in \cite{R}.

A major step forward in the study of toroidal algebras came in \cite{MRY} when Moody, Rao, and Yokonuma gave an alternate presentation of the untwisted types when $\ell = 2$. This presentation is analogous to that of quantum affine algebras given in \cite{D}. It was subsequently generalized in \cite{EM} for an arbitrary positive integer $\ell$.

A few notions of twisted toroidal algebras have been studied \cite{BR1, BR2, FJ, V} with differing restrictions on the twisting automorphism(s); some of these notions are special cases of others. The current work focuses on the twisted toroidal algebras studied in \cite{FJ} which can be viewed as fixed point subalgebras of untwisted toroidal algebras with respect to a Dynkin diagram automorphism (the Dynkin diagram automorphism on a finite type Kac-Moody algebra is extended in a natural way). Only recently has an MRY-like presentation of these algebras been given \cite{JMM, JMM2}.

Two important classes of representations of affine and toroidal algebras are the fermionic and bosonic representations; they view the algebra elements as quadratic operators on a Clifford or Weyl module, respectively, and are distinguished by whether the operators anticommute or commute, respectively. The operators are known as free fields and can be studied through the tools of vertex algebras \cite{K2}. Their study in regards to affine algebras include such seminal works as \cite{KP, F1, F2} and, most notably for the current paper, \cite{FF}. Since that time, numerous authors have used similar techniques to give representations of various affine \cite{KV1, KV2}, toroidal \cite{JM, JMT, JMX, JX}, and related algebras \cite{L, G}. 

In the recent \cite{JMM2}, the aforementioned MRY-like presentation was used to give fermionic representations of the twisted toroidal algebras of \cite{FJ} with $\ell = 2$ of types $\A, \At, \D, \Df$, using similar techniques to \cite{FF, JM}. It included a fermionic representation of the twisted affine algebra $\Df^{(3)}$ for the first time (cf. \cite{F1, FF, KP}). The main result of the current work is to give bosonic representations of the same algebras studied in \cite{JMM2}. In so doing, the current work also contains the first bosonic free field representation of $\Df^{(3)}$.

The paper is organized as follows. In Section \ref{twtorsec} we recall the twisted 2-toroidal Lie algebras of \cite{FJ}, followed by the MRY-like presentation of \cite{JMM2} in Section \ref{MRYPresn}. In Section \ref{bosrepn} we give the main result: the construction of the bosonic free field representations of these algebras.

Throughout this paper, all algebras, spans, tensor products, vector spaces, etc. will be over the field of complex numbers $\C$ unless indicated otherwise.

\section{Twisted 2-Toroidal Lie Algebras}\label{twtorsec}

We begin by recalling the construction of the twisted 2-toroidal Lie algebras \cite{FJ}.

Let $\g$ be the finite dimensional simple Lie algebra of type $A_{2n-1} , (n \geq 3)$, $A_{2n} , (n \geq 2)$, $D_{n+1} ,  (n \geq 2)$, or $D_4$. \footnote{The algebras denoted $D_{n+1}$ for $n=3$ and $D_4$ will be distinguished by their respective values for $r$ defined below; $r=2$ in the former, and $r=3$ in the latter.} The Chevalley generators of $\g$ will be denoted by $\{e_i', f_i', h_i' \mid 1 \leq i \leq N\}$ where $N = 2n - 1, 2n, n+1, 4,$ respectively; $\h' = \text{span}\{h_i' \mid 1 \leq i \leq N\}$ is the Cartan subalgebra of $\g$. Denote the simple roots by $\{\alpha_i' \mid 1 \leq i \leq N\} \subset \h'^*$, the set of roots of $\g$ by $\Delta$, and the root lattice by $Q$. Then $\al_j'(h_i') = a_{ij}'$ where $A' = (a_{ij}')_{i,j = 1}^N$ is the Cartan matrix associated with $\g$.

It is well-known \cite{K1} that $\g$ possesses a nondegenerate symmetric invariant bilinear form which will be denoted $( \cdot | \cdot )$. For $\g$ of type $A_{2n-1}$, $A_{2n}$, $D_{n+1}$, or $D_4$, this form can be realized by defining $( x | y ) = tr(xy), tr(xy),  \frac{1}{2}tr(xy), \frac{1}{2}tr(xy)$, respectively, for all $x, y \in \g$. Then $( h_i' | h_i' ) = 2$ with $1 \leq i \leq N$. Since the Lie algebra $\g$ is simply-laced, we can identify the invariant form on $\h'$ to that on the dual space $\h'^*$ and normalize the inner product by $( \al | \al ) = 2, \al \in \Delta$.

Let $\si$ be the Dynkin diagram automorphism of order $r = 2, 2, 2, 3$ respectively, defined as follows:
\begin{eqnarray*}
&&\sigma(h_i')=h_{N-i+1}', i=1, \cdots , N, \ \ \mbox{for type } \ A_{2n-1}  \ \text{or} \ A_{2n}. \nonumber \\
&&\sigma(h_i')=h_i', i=1, \cdots , N-2; \sigma(h_n')=h_{n+1}',
\ \ \mbox{for type} \ D_{n+1}. \nonumber\\
&&\sigma(h'_1, h'_2, h'_3, h'_4)=(h'_3, h'_2, h'_4, h'_1)\ \ \mbox{for type} \
D_4.\nonumber
\end{eqnarray*}
Let $\sigma$ act on $\{ 1, 2, \ldots , N \}$ analogously: $\sigma(i)=j \Leftrightarrow \sigma(h_i')=h_j'$.
Then $\g$ can be decomposed as a ${\Z}/r{\Z}$-graded Lie algebra:
\begin{equation}\label{findimldecomp}
{\g}={\g}_0\oplus \cdots\oplus{\g}_{r-1},
\end{equation}
where ${\g}_i=\{ x\in {\g} \mid \sigma(x)=\omega^i x\}$ and $\omega=e^{2\pi\sqrt{-1}/r}$. It is known (see \cite[Proposition 8.3]{K1}) that the fixed point subalgebra ${\g}_0$ is the simple Lie algebra of type $C_n$, $B_n$, $B_n$, $G_2$ respectively. Let $I = \{1, 2, \cdots , n\}$ for each $\g$, noting that $n=2$ for $\g = D_4$.
The Chevalley generators $\{e_i, f_i, h_i \mid i \in I\}$ of $\g_0$ are given by:
\begin{eqnarray*}
&&e_i=e_i', f_i=f_i', h_i=h_i', \ \mbox{if } \sigma(i)=i; \\
&&e_i=\sum_{j=0}^{r-1}e'_{\sigma^j(i)}, \
f_i=\sum_{j=0}^{r-1}f'_{\sigma^j(i)}, \
h_i=\sum_{j=0}^{r-1}h'_{\sigma^j(i)}, \ \mbox{if } \sigma(i)\neq i, \text{ except } i=n \text{ for } A_{2n}. \\
&&e_n=\sqrt{2}(e'_n+e'_{n+1}),
f_n=\sqrt{2}(f'_n+f'_{n+1}),
h_n=2(h'_n+h'_{n+1}) \ \mbox{for } A_{2n}.
\end{eqnarray*}

The Cartan subalgebra of $\g_0$ is $\h_0 = \text{span}\{h_i \mid i \in I\}$ and the simple roots $\{\al_i \mid i \in I\}\subset \h_0^*$ are given by:
$$
\al_i = \frac{1}{r}\sum_{j=0}^{r-1}\al'_{\si^{j}(i)}.
$$

Let $A = (a_{ij})_{i,j \in I}$ be the Cartan matrix for $\g_0$. Then we have
\begin{equation}
(\alpha_i|\alpha_j)=d_ia_{ij},
\ \ \mbox{for all} \ i,j\in I
\end{equation}
where
\begin{equation}\label{symmconst}
\displaystyle (d_1, \cdots , d_n) =
\begin{cases}
 \left(\frac{1}{2}, \cdots , \frac{1}{2}, 1 \right) & (A_{2n-1}); \\
 \left(\frac{1}{2}, \cdots , \frac{1}{2}, \frac{1}{4} \right) & (A_{2n}); \\
 \left(1, \cdots, 1, \frac{1}{2} \right) & (D_{n+1}); \\
 \left(\frac{1}{3}, 1 \right) & (D_4).
\end{cases}
\end{equation}


Denote
\begin{center}
$\theta^0 =
\begin{cases}
\a'_1+\cdots+\a'_{2n-2}  & (A_{2n}); \\
\a'_1+\cdots+\a'_{2n} & (A_{2n}); \\
\a'_1+\a'_2+\cdots+\a'_n  & (D_{n+1}); \\
\a'_1+\a'_2+\a'_3 & (D_4).
\end{cases}$
\end{center}
Note that $\theta^0 \in \Delta$ for $\g$. Let $\{ e'_{\theta^0}, f'_{\theta^0}, h'_{\theta^0} \}$ be the $\mathfrak{sl}_2$-triplet associated to $\theta^0$ having bracket $[h'_{\theta^0}, e'_{\theta^0}] = 2 e'_{\theta^0}, [h'_{\theta^0}, f'_{\theta^0}] = -2 f'_{\theta^0}$ and $h'_{\theta^0} = [e'_{\theta^0} , f'_{\theta^0}]$.

Define $\tilde{I} = I \cup \{0\}$ and extend the Cartan matrix $A = (a_{ij})_{i,j \in I}$ for $\g_0$ to the Cartan matrix $\tilde{A} = (a_{ij})_{i,j \in \tilde{I}}$ for the twisted affine algebra $\hat{\g}$, where $\hat{\g}$ is of type $A_{2n-1}^{(2)}, A_{2n}^{(2)}, D_{n+1}^{(2)}, D_4^{(3)}$, respectively, when $\g$ is of type $A_{2n-1}, A_{2n}, D_{n+1}, D_4$, respectively. Let $\{ \alpha_i \mid i \in \tilde{I} \}$, $\hat{Q}$, $\delta$, and $\hat{\Delta}$ denote the simple roots, root lattice, null root, and set of roots, respectively, for $\hat{\g}$.

Let $\cA = \C[s, s^{-1}, t, t^{-1}]$ be the ring of Laurent polynomials in the commuting variables $s, t$ and $L(\g) = \g \otimes \cA$  be the multi-loop algebra with the following Lie bracket:

$$
[x \otimes s^j t^k , y \otimes s^l t^m] = [x, y] \otimes s^{j+l}t^{k+m},
$$
for all $x, y \in \g$, and $j, k, l, m \in \Z$. For $j \in \Z$ we define $0 \leq \overline{j} < r$ such that $j \equiv \overline{j} \ \mbox{mod} \ r$, and define $\g_j = \g_{\overline{j}}$. We extend the automorphism $\si$ of $\g$ to an automorphism $\overline{\si}$ of $L(\g)$ by defining:
$$
\overline\si(x \otimes s^j t^k) = \omega^{-k} \si(x) \otimes s^j t^k
$$
where $x \in \g, j, k \in \Z$. We denote by $L(\g, \si)$ the fixed point subalgebra of $L(\g)$ with respect to $\overline\si$. Note that $L(\g, \si)$ is $\Z$-graded as follows:
$$
\displaystyle L(\g, \si) = \bigoplus_{m \in \Z} L(\g, \si)_m,
$$
where $L(\g, \si)_m = \g_m \otimes \cA_m, \cA_m = \mbox{span}\{s^j t^m \mid j \in \Z\}=t^m\C[s, s^{-1}]$.

Set $F = \cA \otimes \cA$. Then the action $a(b_1 \otimes b_2) = ab_1 \otimes b_2 = (b_1 \otimes b_2)a$ for all $a, b_1, b_2 \in \cA$ gives $F$ the structure of a two-sided $\cA$-module. Let $G$ be the $\cA$-submodule of $F$ generated by $\{1 \otimes ab -a\otimes b -b \otimes a \mid a, b \in \cA\}$. The $\cA$- quotient module $\Omega_{\cA} = F/G$ is called the $\cA$ - module of K{\"a}hler differentials, with the differential map being the canonical quotient map $d : \cA \longrightarrow \Omega_{\cA}$ given by $da = (1 \otimes a) + G$ for $a \in \cA$. Let $\cK' = \Omega_{\cA}/d\cA$ and let $- : \Omega_{\cA} \longrightarrow \cK'$ be the canonical linear map. Since $\overline{d(ab)} = 0$, we have $\overline{a(db)} = - \overline{(da)b} = - \overline{b(da)}$ for all  $a, b \in \cA$. Then $\cK' = span\{\overline{bda} \mid a, b \in \cA\}$. Set $\cK = span\{\overline{bda} \mid a \in \cA_k, b \in \cA_l , k+l \equiv 0 \text{ (mod } r)\}$ which is a subalgebra of $\cK'$. We note that $\{\overline{s^{j-1}t^mds}, \overline{s^jt^{-1}dt}, \overline{s^{-1}ds}\mid j \in \Z, m \in \Z_{\neq 0}\}$ is a basis for $\cK$ and the following relations can be checked by direct computation.
\begin{equation}\label{kahlercalc}
\begin{split}
\overline{s^{\ell}ds^k} & = \delta_{k, -\ell} k \overline{s^{-1}ds}, \\
\overline{s^{\ell}t^{-1}d(s^{k}t)} & = \delta_{k, -\ell} k \overline{s^{-1}ds} + \overline{s^{k + \ell}t^{-1}dt}.
\end{split}
\end{equation}

Define
$$
T(\g) = L(\g , \si) \oplus \cK,
$$
with the Lie bracket given by
\begin{equation*}
[x\otimes a, y\otimes b] = [x, y]\otimes ab + (x | y)\overline{bda},
\end{equation*}
and $\cK$ central, where $x \in \g_i, y \in \g_j, a \in \cA_i, b \in \cA_j$ for $i, j \in \Z$. It is shown in \cite[Theorem 2.1]{FJ} that $T(\g)$ with the canonical projection map $\eta : T(\g) \rightarrow L(\g, \si)$ is the universal central extension of $L(\g, \si)$. $T(\g)$ is called the twisted 2-toroidal Lie algebra of type $\g$.

The elements $c_0 = \overline{s^{-1}ds}, c_1= \overline{t^{-1}dt} \in \cK$ are called the degree zero central elements. A representation of $T(\g)$ is said to be of level $(k_0, k_1)$ for $k_0, k_1 \in \C$ if $c_0$ acts as $k_0(id)$ and $c_1$ acts as
$k_1(id)$.

\section{MRY-Like Presentation of Twisted 2-Toroidal Algebras}\label{MRYPresn}

In \cite{JMM2}, another presentation of $T(\g)$ is given which is similar to that in \cite{MRY} of \emph{untwisted} 2-toroidal Lie algebras; this presentation is in turn based upon the realization of quantum affine algebras given in \cite{D}. Here we recall the essentials of this construction (for further details, consult \cite{JMM2}).

Let $t(\g)$ be the Lie algebra generated by symbols
$$\not{c},
 \al_m(k) \text{ and } X(\pm \al_m, k),$$
with $m \in \tilde{I}$ and $k \in \Z$, and satisfying the following relations:
\begin{enumerate}
  \item $[\al_0(k), \al_0(l)] =
    \begin{cases}
      2r k \delta_{k,-\ell} \not{c} & (A_{2n-1}, D_{n+1}, D_4); \\
      2k \delta_{k,-\ell} \not{c} & (A_{2n}).
    \end{cases}$ \\
  \item $[\al_i(k), \al_j(l)] =
    \begin{cases}
      r a_{ij} k \delta_{k,-\ell} \not{c} & (A_{2n-1}, A_{2n}, D_4); \\
      a_{ij} k \delta_{k,-\ell} \not{c} & (D_{n+1})
    \end{cases}$ \\
where $i \in \tilde{I}, j \in I$ with $i \leq j$ and $(i, j) \neq (n-1, n), (n, n)$.
  \item $[\al_i(k), \al_j(l)] =
    \begin{cases}
      a_{ij} k \delta_{k,-\ell} \not{c} & (A_{2n-1}, D_4); \\
      4 a_{ij} k \delta_{k,-\ell} \not{c} & (A_{2n}); \\
      2 a_{ij} k \delta_{k,-\ell} \not{c} & (D_{n+1})
    \end{cases}$ \\
where $(i, j) = (n-1, n)$ or $(i, j)= (n, n)$.
 \item $[\al_i(k), X(\pm \al_j, l)] =
      \pm a_{ij} X(\pm \al_j, k+l)$
where $i,j \in \tilde{I}$.
 \item $[X(\pm \al_i, k) , X(\pm \al_i, l) ] = 0 $
where $i \in \tilde{I}$.
 \item $[X(\al_i, k) , X(-\al_j, l) ] = \\
    \begin{cases}
      \delta_{i,j} \big\{ \al_i(k+l) + \big(r - \delta_{i,n} (r-1) \big) k \delta_{k,-\ell} \not{c} \big \} & (A_{2n-1}, D_4); \\
      \delta_{i,j} \big\{ \al_i(k+l) + \big( r \left(1+\delta_{i,n}(r-1) \right) - \delta_{i,0}(r-1) \big) k \delta_{k,-\ell} \not{c} \big \} & (A_{2n}); \\
      \delta_{i,j} \big\{ \al_i(k+l) + \big(1 + (\delta_{i,0}+\delta_{i,n}) (r-1) \big) k \delta_{k,-\ell} \not{c} \big \} & (D_{n+1})
    \end{cases}$  \\
where $i,j \in \tilde{I}$.
 \item $\text{ad} X(\pm \al_i , k_2) X(\pm \al_j , k_1) = 0$ for $i,j \in \tilde{I}$ with $i \neq j$ and $a_{ij} = 0$.
 \item $\text{ad} X(\pm \al_i, k_3) \text{ad} X(\pm \al_i , k_2) X(\pm \al_j , k_1) = 0$ for $i,j \in \tilde{I}$ with $i \neq j$ and $a_{ij} = -1$.
 \item $\text{ad} X(\pm \al_i, k_4) \text{ad} X(\pm \al_i , k_3) \text{ad} X(\pm \al_i , k_2) X(\pm \al_j , k_1) = 0$ for $i,j \in \tilde{I}$ with $i \neq j$ and $a_{ij} = -2$.
 \item $\text{ad} X(\pm \al_i, k_5) \text{ad} X(\pm \al_i, k_4) \text{ad} X(\pm \al_i , k_3) \text{ad} X(\pm \al_i , k_2) X(\pm \al_j , k_1) = 0$ for $i,j \in \tilde{I}$ with $i \neq j$ and $a_{ij} = -3$.
\end{enumerate}
In addition, $\not{c}$ is central.

Now let $z,w$ be formal variables. We define formal power series with coefficients from $t(\g)$:
 $$
 \al_i(z)=\sum_{k\in \Z}\al_i(k)z^{-k-1},
\qquad X(\pm \al_i,z)=\sum_{k\in \Z}X(\pm \al_i, k)z^{-k-1},
 $$
for $i \in \tilde{I}$.
We will use the delta function $\delta(z-w)=\sum_{k\in\Z}w^k z^{-k-1}$. Since $\frac1{z-w} = \sum_{k=0}^{\infty} z^{-k-1} w^k$ in the region $|z|>|w|$, then:
\begin{align*}
\delta(z-w)&=\iota_{z, w}\left(\frac{1}{z-w} \right)+\iota_{w, z}\left( \frac{1}{w-z} \right),\\
\partial_w\delta(z-w)&=
\iota_{z, w}\left( \frac{1}{(z-w)^{2}} \right)-\iota_{w, z}\left( \frac{1}{(w-z)^{2}} \right),
\end{align*}
where $\iota_{z, w}$ means the expansion in the region $|z|>|w|$. For simplicity in the following we will drop $\iota_{z, w}$ if it is clear from context.

We may write the above defining relations of $t(\g)$ as power series by collecting components as follows.
\begin{equation}\label{twrelationsps}
\textbf{Relations of } t(\g) \textbf{ as Power Series}
\end{equation}
\begin{enumerate}
  \item $[\al_0(z), \al_0(w)] =
    \begin{cases}
      2r \partial_w\delta(z-w) \not{c} & (A_{2n-1}, D_{n+1}, D_4); \\
      2 \partial_w\delta(z-w) \not{c} & (A_{2n}).
    \end{cases}$ \\
  \item $[\al_i(z), \al_j(w)] =
    \begin{cases}
      r a_{ij} \partial_w\delta(z-w) \not{c} & (A_{2n-1}, A_{2n}, D_4); \\
      a_{ij} \partial_w\delta(z-w) \not{c} & (D_{n+1})
    \end{cases}$ \\
where $i \in \tilde{I}, j \in I$ with $i \leq j$ and $(i, j) \neq (n-1, n), (n, n)$.
  \item $[\al_i(z), \al_j(w)] =
    \begin{cases}
      a_{ij} \partial_w\delta(z-w) \not{c} & (A_{2n-1}, D_4); \\
      4 a_{ij} \partial_w\delta(z-w) \not{c} & (A_{2n}); \\
      2 a_{ij} \partial_w\delta(z-w) \not{c} & (D_{n+1})
    \end{cases}$ \\
where $(i, j) = (n-1, n)$ or $(i, j)= (n, n)$.
 \item $[\al_i(z), X(\pm \al_j, w)] =
      \pm a_{ij} X(\pm \al_j, z) \delta(z-w)$
where $i,j \in \tilde{I}$.
 \item $[X(\pm \al_i, z) , X(\pm \al_i, w) ] = 0 $
where $i \in \tilde{I}$.
 \item $[X(\al_i, z) , X(-\al_j, w) ] = \\
    \begin{cases}
      \delta_{i,j} \big\{ \al_i(z) \delta(z-w) + \big(r - \delta_{i,n} (r-1) \big) \partial_w\delta(z-w) \not{c} \big \} \qquad (A_{2n-1}, D_4); \\
      \delta_{i,j} \big\{ \al_i(z) \delta(z-w) + \big( r \left(1+\delta_{i,n}(r-1) \right) - \delta_{i,0}(r-1) \big) \partial_w\delta(z-w) \not{c} \big \} \, (A_{2n}); \\
      \delta_{i,j} \big\{ \al_i(z) \delta(z-w) + \big(1 + (\delta_{i,0}+\delta_{i,n}) (r-1) \big) \partial_w\delta(z-w) \not{c} \big \} \qquad (D_{n+1})
    \end{cases}$  \\
where $i,j \in \tilde{I}$.
 \item $\text{ad} X(\pm \al_i , z_2) X(\pm \al_j , z_1) = 0$ for $i,j \in \tilde{I}$ with $i \neq j$ and $a_{ij} = 0$.
 \item $\text{ad} X(\pm \al_i, z_3) \text{ad} X(\pm \al_i , z_2) X(\pm \al_j , z_1) = 0$ for $i,j \in \tilde{I}$ with $i \neq j$ and $a_{ij} = -1$.
 \item $\text{ad} X(\pm \al_i, z_4) \text{ad} X(\pm \al_i , z_3) \text{ad} X(\pm \al_i , z_2) X(\pm \al_j , z_1) = 0$ for $i,j \in \tilde{I}$ with $i \neq j$ and $a_{ij} = -2$.
 \item $\text{ad} X(\pm \al_i, z_5) \text{ad} X(\pm \al_i, z_4) \text{ad} X(\pm \al_i , z_3) \text{ad} X(\pm \al_i , z_2) X(\pm \al_j , z_1) = 0$ for $i,j \in \tilde{I}$ with $i \neq j$ and $a_{ij} = -3$.
\end{enumerate}

For any power series $a(z) = \sum_{k \in \Z} a(k) z^{-k-1}$, note that $a(z) \delta(z-w) = a(w) \delta(z-w)$. Indeed, $a(z) \delta(z-w) = \sum_{k \in \Z} a(k) z^{-k-1} \sum_{\ell \in \Z} w^{\ell} z^{-\ell-1} =\\ \sum_{k,\ell \in \Z} a(k) z^{-k-\ell-2} w^{\ell}$ and setting $k'=k+\ell+1$ gives $\sum_{k,k' \in \Z} a(k) z^{-k'-1} w^{k'-k-1} = \sum_{k \in \Z} a(k) w^{-k-1} \sum_{k' \in \Z} w^{k'} z^{-k'-1} =  a(w) \delta(z-w)$.

Define a map $\pi : t(\g) \longrightarrow L(\g, \si)$ as follows. In types $A_{2n-1}, D_{n+1}, D_4$,
\begin{equation*}\label{piADD}
\begin{cases}
 \not{c} \mapsto 0; \\
 \al_{0} (k) \mapsto \sum_{p=0}^{r-1} \sigma^{p} (-h'_{\theta^0}) \otimes s^{k}; \\
 \al_{i} (k) \mapsto \left(1 - \delta_{i,\sigma(i)} \left(1-\frac{1}{r} \right) \right) \sum_{p=0}^{r-1} \sigma^{p} (h'_{i}) \otimes s^{k}; \\
 X(\al_{0}, k) \mapsto \sum_{p=0}^{r-1} - \sigma^{p} (\omega^{r-p} f'_{\theta^0}) \otimes (s^{k} t); \\
 X(-\al_{0}, k) \mapsto \sum_{p=0}^{r-1} - \sigma^{p} (\omega^{p} e'_{\theta^0}) \otimes (s^{k} t^{-1}); \\
 X(\al_{i}, k) \mapsto \left(1 - \delta_{i,\sigma(i)} \left(1-\frac{1}{r} \right) \right) \sum_{p=0}^{r-1} \sigma^{p} (e'_{i}) \otimes s^{k}; \\
 X(-\al_{i}, k) \mapsto \left(1 - \delta_{i,\sigma(i)} \left(1-\frac{1}{r} \right) \right) \sum_{p=0}^{r-1} \sigma^{p} (f'_{i}) \otimes s^{k}.
\end{cases}
\end{equation*}

In type $A_{2n}$,
\begin{equation*}\label{piA2}
\begin{cases}
 \not{c} \mapsto 0; \\
 \al_{0} (k) \mapsto -h'_{\theta^0} \otimes s^{k}; \\
 \al_{i} (k) \mapsto \left(1 + \delta_{i,n} \right) \sum_{p=0}^{r-1} \sigma^{p} (h'_{i}) \otimes s^{k}; \\
 X(\al_{0}, k) \mapsto - f'_{\theta^0} \otimes (s^{k} t); \\
 X(-\al_{0}, k) \mapsto - e'_{\theta^0} \otimes (s^{k} t^{-1}); \\
 X(\al_{i}, k) \mapsto \left(1 + \delta_{i,n} \left(\sqrt{2}-1 \right) \right) \sum_{p=0}^{r-1} \sigma^{p} (e'_{i}) \otimes s^{k}; \\
 X(-\al_{i}, k) \mapsto \left(1 + \delta_{i,n} \left(\sqrt{2}-1 \right) \right) \sum_{p=0}^{r-1} \sigma^{p} (f'_{i}) \otimes s^{k}.
\end{cases}
\end{equation*}

The following proposition (see \cite[Theorem 3.3]{JMM2}) shows that $t(\g)$ is a realization of the twisted toroidal Lie algebra $T(\g)$. We will call it the MRY-like presentation of $T(\g)$.

\begin{proposition}\label{MRYJMM2}
The pair $(t(\g), \pi)$ is the universal central extension of $L(\g, \si)$. Hence $t(\g) \cong T(\g)$.
\end{proposition}
\begin{proof}
\end{proof}

\section{Bosonic Representation}\label{bosrepn}

In this section we will use the MRY-like presentation of $T(\g)$ from Proposition \ref{MRYJMM2} to give a bosonic free field representation of $t(\g)$ for $\g = A_{2n-1}, A_{2n}, D_{n+1}, D_4$.

Consider the vector space  $\C^{n+4}$ with basis $\{\varepsilon_i \mid i = 0,1, \cdots , n+3\}$. This basis is orthonormal with respect to the inner product $(\cdot | \cdot )$ defined by
\begin{center}
$(\varepsilon_i | \varepsilon_j) = \delta_{ij}$.
\end{center}
Consider the lattice $P_0 = \bigoplus_{i=1}^{n+2}\Z \varepsilon_i$ and set $c = \frac{1}{\sqrt{2}} (\varepsilon_0 + i \varepsilon_{n+3})$ and $d = \frac{1}{\sqrt{2}} (\varepsilon_0 - i \varepsilon_{n+3})$. Then $( c | c) = 0 = ( d | d)$ and $( c | d) = 1$.

Recall the decomposition (\ref{findimldecomp}) of $\g$ where $r=2$ when $\g=A_{2n-1}, A_{2n}, D_{n+1}$ and $r=3$ when $\g=D_4$. The simple roots of the fixed point subalgebra $\g_0$ of $\g$ can be realized via the $\varepsilon_i$ elements as follows, with $i \in I \backslash \{n\}$ for $\g=A_{2n-1}, A_{2n}, D_{n+1}$.

\begin{itemize}
\item $\alpha_i = \frac{1}{\sqrt{2}}(\varepsilon_i - \varepsilon_{i+1}), \alpha_n=\sqrt{2}\varepsilon_n$, for $\g = A_{2n-1}$;
\item $\alpha_i = \frac{1}{\sqrt{2}}(\varepsilon_i - \varepsilon_{i+1}), \alpha_n=\frac{1}{\sqrt{2}}\varepsilon_n$, for $\g = A_{2n}$;
\item $\alpha_i = \varepsilon_i - \varepsilon_{i+1}, \alpha_n=\varepsilon_n$, for $\g = D_{n+1}$;
\item $\alpha_1 = \frac{1}{\sqrt{3}}(\varepsilon_1 - \varepsilon_2), \alpha_2 = \frac{1}{\sqrt{3}}(-\varepsilon_1 + 2\varepsilon_2 - \varepsilon_3)$ for $\g = D_4$.
\end{itemize}

It is known (see \cite[Proposition 8.3]{K1}) that $\g_1$ is an irreducible $\g_0$-module. The highest weight of this representation can be realized via the $\varepsilon_i$ elements as

\begin{center}
$\displaystyle \theta_0 := \frac{1}{r}\sum_{j=0}^{r-1} \sigma^j (\theta^0) =
\begin{cases}
 \frac{1}{\sqrt{2}}(\varepsilon_1 + \varepsilon_2) & (A_{2n-1}); \\
 \sqrt{2}\varepsilon_1 & (A_{2n}); \\
 \varepsilon_1 & (D_{n+1}); \\
 \frac{1}{\sqrt{3}}(\varepsilon_1 - \varepsilon_3) & (D_4).
\end{cases}$
\end{center}

Now define
\begin{center}
$\displaystyle \beta :=
\begin{cases}
 -\sqrt{2}c +\varepsilon_1  & (A_{2n-1}); \\
 -\frac{1}{\sqrt{2}}c + \varepsilon_1 & (A_{2n}); \\
 -c+\varepsilon_1 & (D_{n+1}); \\
 -\sqrt{3}c + \varepsilon_1 & (D_4)
\end{cases}$
\end{center}
and set
\begin{center}
$\alpha_0 := c - \theta_0 =
\begin{cases}
  -\frac{1}{\sqrt{2}}(\beta+\varepsilon_2) & (A_{2n-1}); \\
  -\sqrt{2}\beta  & (A_{2n}); \\
  -\beta & (D_{n+1}); \\
  - \frac{1}{\sqrt{3}}(\beta -\varepsilon_3) & (D_4).
\end{cases}$
\end{center}
Direct computation shows that the $\{\alpha_i \mid i \in \tilde{I} \}$ form the set of simple roots of the twisted affine Lie algebra $\hat{\g}$ with the GCM
$\tilde{A} = (a_{ij})_{i,j \in \tilde{I}}$. The symmetric nondegenerate invariant bilinear form is given by
\begin{equation}
(\alpha_i|\alpha_j)=d_ia_{ij},
\ \ \mbox{for all} \ i,j\in \tilde{I},
\end{equation}
where
\begin{center}
$\displaystyle (d_0,d_1, \cdots , d_n) =
\begin{cases}
 (\frac{1}{2},\frac{1}{2}, \cdots , \frac{1}{2}, 1)& (A_{2n-1}); \\
 (1, \frac{1}{2}, \cdots , \frac{1}{2}, \frac{1}{4}) & (A_{2n}); \\
 (\frac{1}{2}, 1, \cdots, 1, \frac{1}{2}) & (D_{n+1}); \\
 (\frac{1}{3}, \frac{1}{3}, 1) & (D_4).
\end{cases}$
\end{center}
Notice that the symmetrization constants $(d_0, \ldots , d_n)$ simply extend the constants (\ref{symmconst}) by adding the appropriate $d_0$ in each type. Furthermore, observe that $(c | \alpha_i) = 0 = (d | \alpha_i), i \in \tilde{I}$, so $c$ corresponds to the null root $\delta$ and $d$ corresponds to the dual gradation operator for $\hat{\g}$.

We now introduce a second copy of $\C^{n+4}$ having as a basis $\{\varepsilon_{\overline{i}} \mid i = 0,1, \cdots , n+3\}$. Extend the inner product $(\cdot | \cdot )$ by
\begin{center}
$(\varepsilon_{\overline{i}} | \varepsilon_{\overline{j}}) = \delta_{ij}$ and $(\varepsilon_i | \varepsilon_{\overline{j}}) = 0$
\end{center}
so that the $\{\varepsilon_{\overline{i}} \mid i = 0,1, \cdots , n+3\}$ form an orthonormal basis of $\C^{n+4}$. Form the lattice $\overline{P}_0 = \bigoplus_{i=1}^{n+2}\Z \varepsilon_{\overline{i}}$.  Observe that $(c | x) = 0$ for all $x \in \overline{P}_0$.

Define
\begin{center}
$\displaystyle \overline{\beta} :=
\begin{cases}
 -\sqrt{2}c +\varepsilon_{\overline{1}} & (A_{2n-1}); \\
 -\frac{1}{\sqrt{2}}c + \varepsilon_{\overline{1}} & (A_{2n}); \\
 -c+\varepsilon_{\overline{1}} & (D_{n+1}); \\
 -\sqrt{3}c + \varepsilon_{\overline{1}} & (D_4).
\end{cases}$
\end{center}

Now consider $P := P_0 \oplus \overline{P}_0 \oplus \Z c$ and denote by $P_{\C}$ the vector space $\C \otimes P$ which has $\{c, \varepsilon_i, \varepsilon_{\overline{i}} \mid i=1, 2, \ldots n+2 \}$ as a spanning set. Define $P^*$ by the condition $a \in P \Leftrightarrow a^* \in P^*$ and $P^*_{\C} := \C \otimes P^*$ (and hence $a \in P_{\C} \Leftrightarrow a^* \in P^*_{\C}$). Then the vector space $\mathcal C := P_{\C} \oplus P^*_{\C}$ is a polarization into maximal complementary isotropic subspaces with respect to the antisymmetric bilinear form given by
\begin{equation}\label{antisymmbilform}
-\langle a, b^* \rangle = \langle b^*, a \rangle = (a|b), \quad \langle a, b \rangle = \langle a^*, b^* \rangle=0,
\end{equation}
for all $a, b \in P_{\C}$.

Consider the unital algebra $W(P)$ generated by elements $a(k), a^*(k)$ where $a \in P_{\C}, k \in \mathcal{Z}$, for $\mathcal{Z} = \Z$ or $\mathcal{Z} = \Z + \frac{1}{2}$, with the relation
$$
[a(k), b(l)] = \langle a, b \rangle \delta_{k, -l}
$$
for $a, b\in \mathcal{C}$. $W(P)$ is an infinite dimensional Weyl algebra \cite{JMX}. Let
\begin{center}
$\displaystyle V := \bigotimes_{a} \left( \bigotimes_{k \in \mathcal{Z}_{> 0}} \C [a(-k)] \bigotimes_{k \in \mathcal{Z}_{> 0}} \C [a^* (-k)] \right)$
\end{center}
where the $a$ runs over $\{c, \varepsilon_i, \varepsilon_{\overline{i}} \mid i=1, 2, \ldots n+2 \}$. $V$ is an infinite dimensional representation space, called the Fock space, on which $W(P)$ acts in the usual way: $a(-k)$ acts as a creation operator and $a(k)$ acts as an annihilation operator for $k \in \mathcal{Z}_{> 0}$.

We introduce a notion of normal ordering. For $a(m), b(n) \in W(P)$, define
\begin{center}
$: \! a(m) b(n) \! : =
\begin{cases}
a(m)b(n) & \mbox{if } m < 0; \\
\frac{1}{2} \big(a(m)b(n) + b(n)a(m) \big) & \mbox{if } m = 0; \\
b(n)a(m) & \mbox{if } m > 0.
\end{cases}$
\end{center}
For formal variables $z,w$, we define bosonic fields by collecting components:

$$a(z) = \dsum_{m \in \mathcal{Z}}a(m)z^{-m-\frac{1}{2}} \ \  {\mbox{and}}  \ \  b(w) = \dsum_{n \in \mathcal{Z}}b(n)w^{-n-\frac{1}{2}}.$$
It follows that the normal ordering satisfies the relation
$$
: \! a(z)b(w) \! : = : \! b(w)a(z) \! :.
$$
The normal ordered product of $k > 2$ fields is defined inductively by:
$$
: \! a_1(z_1)a_2(z_2) \cdots a_k(z_k) \! : = : \! a_1(z_1) \left(: \! a_2(z_2) \cdots a_k(z_k) \! : \right) \! :.
$$
The contraction of two fields is defined by
$$
\underbrace{a(z)b(w)}=a(z)b(w) - : \! a(z)b(w) \! :,
$$
which contains all poles for $a(z)b(w)$. 

In order to calculate the bracket among normal ordered products, we make use of the following result which follows from Wick's theorem \cite[Theorem 3.3]{K2}, \cite[Theorem 3.1]{JMX}. The result holds in both cases, $\mathcal{Z} = \Z$ and $\mathcal{Z} = \Z + \frac{1}{2}$.

\begin{proposition}(\cite{JMX}, Proposition 3.3)
\label{bosbracket}
For $a, b, u, v \in \mathcal{C}$ and formal variables $z,w$, we have

$[ : \! a (z) b^* (z) \! : , : \! u (w) v^* (w) \! : ]
 = \langle a , v^* \rangle : \! b^* (z) u (z) \! : \delta (z - w)$ \\
$ + \langle b^* , u \rangle : \! a (z) v^* (z) \! : \delta (z - w)
 + \langle a , v^* \rangle  \langle b^* , u \rangle \partial_w \delta (z - w)$.
\end{proposition}

In the following theorem, using the MRY-like presentation of $T(\g)$ given in Proposition \ref{MRYJMM2}, we give a bosonic representation of the twisted toroidal algebra $t(\g)$, $\g = A_{2n-1}, A_{2n}, D_{n+1}, D_4$, on the Fock space $V$. This is our main result.

\begin{theorem}
Define a map $\rho : t(\g) \rightarrow \text{End}(V)$ as follows for $\g$ of each type, with $1 \leq i \leq n-1$.
For type $\A$,

\begin{itemize}
 \item $\not{c} \mapsto -1$,
 \item $\alpha_0 (z) \mapsto \,  - : \! \beta(z) \beta^*(z) \! : - : \! \varepsilon_2(z) \varepsilon^*_2(z) \! : + : \! \varepsilon_{\overline{1}}(z) \varepsilon^*_{\overline{1}}(z) \! : +  : \! \varepsilon_{\overline{2}}(z) \varepsilon^*_{\overline{2}}(z) \! :$,
 \item $\alpha_i (z) \mapsto \, : \! \varepsilon_i(z) \varepsilon^*_i(z) \! : - : \! \varepsilon_{i+1}(z) \varepsilon^*_{i+1}(z) \! :
- : \! \varepsilon_{\overline{i}}(z) \varepsilon^*_{\overline{i}}(z) \! : + : \! \varepsilon_{\overline{i+1}}(z) \varepsilon^*_{\overline{i+1}}(z) \! :$,
 \item $\alpha_n (z) \mapsto \, : \! \varepsilon_n(z) \varepsilon^*_n(z) \! : - : \! \varepsilon_{\overline{n}}(z) \varepsilon^*_{\overline{n}}(z) \! :$,
 \item $X(\alpha_0, z) \mapsto \, : \! \varepsilon_{\overline{2}}(z) \beta^*(z) \! : + : \! \varepsilon_{\overline{1}}(z) \varepsilon^*_2(z) \! :$,
 \item $X(-\alpha_0, z) \mapsto \, : \! \varepsilon^*_{\overline{2}}(z) \beta(z) \! : + : \! \varepsilon^*_{\overline{1}}(z) \varepsilon_2(z) \! :$,
 \item $X(\alpha_i, z) \mapsto \, : \! \varepsilon_i(z) \varepsilon^*_{i+1}(z) \! : + : \! \varepsilon_{\overline{i+1}}(z) \varepsilon^*_{\overline{i}}(z) \! :$,
 \item $X(-\alpha_i, z) \mapsto \, : \! \varepsilon^*_i(z) \varepsilon_{i+1}(z) \! : + : \! \varepsilon^*_{\overline{i+1}}(z) \varepsilon_{\overline{i}}(z) \! :$,
 \item $X(\alpha_n, z) \mapsto \, : \! \varepsilon_n(z) \varepsilon^*_{\overline{n}}(z) \! :$,
 \item $X(-\alpha_n, z) \mapsto \, : \! \varepsilon^*_n(z) \varepsilon_{\overline{n}}(z) \! :$.
\end{itemize}

For type $\At$,

\begin{itemize}
 \item $\not{c} \mapsto -1$,
 \item $\alpha_0 (z) \mapsto \, : \! \varepsilon_{\overline{1}}(z) \varepsilon^*_{\overline{1}}(z) \! : - : \! \beta(z) \beta^*(z) \! :$,
 \item $\alpha_i (z) \mapsto \, : \! \varepsilon_i(z) \varepsilon^*_i(z) \! : - : \! \varepsilon_{i+1}(z) \varepsilon^*_{i+1}(z) \! :
- : \! \varepsilon_{\overline{i}}(z) \varepsilon^*_{\overline{i}}(z) \! : + : \! \varepsilon_{\overline{i+1}}(z) \varepsilon^*_{\overline{i+1}}(z) \! :$,
 \item $\alpha_n (z) \mapsto \, 2 (: \! \varepsilon_n(z) \varepsilon^*_n(z) \! : - : \! \varepsilon_{\overline{n}}(z) \varepsilon^*_{\overline{n}}(z) \! :)$,
 \item $X(\alpha_0, z) \mapsto \, : \! \beta^*(z) \varepsilon_{\overline{1}}(z) \! :$,
 \item $X(-\alpha_0, z) \mapsto \, : \! \beta(z) \varepsilon^*_{\overline{1}}(z) \! :$,
 \item $X(\alpha_i, z) \mapsto \, : \! \varepsilon_i(z) \varepsilon^*_{i+1}(z) \! : + : \! \varepsilon_{\overline{i+1}}(z) \varepsilon^*_{\overline{i}}(z) \! :$,
 \item $X(-\alpha_i, z) \mapsto \, : \! \varepsilon^*_i(z) \varepsilon_{i+1}(z) \! : + : \! \varepsilon^*_{\overline{i+1}}(z) \varepsilon_{\overline{i}}(z) \! :$,
 \item $X(\alpha_n, z) \mapsto \, \sqrt{2} (: \! \varepsilon_n(z) \varepsilon^*_{n+1}(z) \! : - : \! \varepsilon_{n+1}(z) \varepsilon^*_{\overline{n}}(z) \! :)$,
 \item $X(-\alpha_n, z) \mapsto \, \sqrt{2} (: \! \varepsilon^*_n(z) \varepsilon_{n+1}(z) \! : - : \! \varepsilon^*_{n+1}(z) \varepsilon_{\overline{n}}(z) \! :)$.
\end{itemize}

For type $\D$,

\begin{itemize}
 \item $\not{c} \mapsto -2$,
 \item $\alpha_0 (z) \mapsto \, 2 (: \! \varepsilon_{\overline{1}}(z) \varepsilon^*_{\overline{1}}(z) \! : - : \! \beta(z) \beta^*(z) \! :)$,
 \item $\alpha_i (z) \mapsto \, : \! \varepsilon_i(z) \varepsilon^*_i(z) \! : - : \! \varepsilon_{i+1}(z) \varepsilon^*_{i+1}(z) \! :
- : \! \varepsilon_{\overline{i}}(z) \varepsilon^*_{\overline{i}}(z) \! : + : \! \varepsilon_{\overline{i+1}}(z) \varepsilon^*_{\overline{i+1}}(z) \! :$,
 \item $\alpha_n (z) \mapsto \, 2(: \! \varepsilon_n(z) \varepsilon^*_n(z) \! : - : \! \varepsilon_{\overline{n}}(z) \varepsilon^*_{\overline{n}}(z) \! :)$,
 \item $X(\alpha_0, z) \mapsto \, \sqrt{2} (: \! \beta^*(z) \varepsilon_{n+2}(z) \! : - : \! \varepsilon^*_{n+2}(z) \varepsilon_{\overline{1}}(z) \! :)$,
 \item $X(-\alpha_0, z) \mapsto \, \sqrt{2} (: \! \beta(z) \varepsilon^*_{n+2}(z) \! : - : \! \varepsilon_{n+2}(z) \varepsilon^*_{\overline{1}}(z) \! :)$,
 \item $X(\alpha_i, z) \mapsto \, : \! \varepsilon_i(z) \varepsilon^*_{i+1}(z) \! : + : \! \varepsilon_{\overline{i+1}}(z) \varepsilon^*_{\overline{i}}(z) \! :$,
 \item $X(-\alpha_i, z) \mapsto \, : \! \varepsilon^*_i(z) \varepsilon_{i+1}(z) \! : + : \! \varepsilon^*_{\overline{i+1}}(z) \varepsilon_{\overline{i}}(z) \! :$,
 \item $X(\alpha_n, z) \mapsto \, \sqrt{2} (: \! \varepsilon_n(z) \varepsilon^*_{n+1}(z) \! : - : \! \varepsilon_{n+1}(z) \varepsilon^*_{\overline{n}}(z) \! :)$,
 \item $X(-\alpha_n, z) \mapsto \, \sqrt{2} (: \! \varepsilon^*_n(z) \varepsilon_{n+1}(z) \! : - : \! \varepsilon^*_{n+1}(z) \varepsilon_{\overline{n}}(z) \! :)$.
\end{itemize}

For type $\Df$,

\begin{itemize}
 \item $\not{c} \mapsto -2$,
 \item $\alpha_0 (z) \mapsto \, 2 : \! \varepsilon_{\overline{1}}(z) \varepsilon^*_{\overline{1}}(z) \! : - 2 : \! \beta(z) \beta^*(z) \! : 
 + : \! \varepsilon_{\overline{2}}(z) \varepsilon^*_{\overline{2}}(z) \! : - : \! \varepsilon_2(z) \varepsilon^*_2(z) \! :
 + : \! \varepsilon_{\overline{3}}(z) \varepsilon^*_{\overline{3}}(z) \! : - : \! \varepsilon_3(z) \varepsilon^*_3(z) \! :$,
 \item $\alpha_1 (z) \mapsto \, : \! \varepsilon_1(z) \varepsilon^*_1(z) \! : - : \! \varepsilon_{\overline{1}}(z) \varepsilon^*_{\overline{1}}(z) \! : 
 - : \! \varepsilon_2(z) \varepsilon^*_2(z) \! : + : \! \varepsilon_{\overline{2}}(z) \varepsilon^*_{\overline{2}}(z) \! :
+ 2 : \! \varepsilon_3(z) \varepsilon^*_3(z) \! :  - 2 : \! \varepsilon_{\overline{3}}(z) \varepsilon^*_{\overline{3}}(z) \! :$,
 \item $\alpha_2 (z) \mapsto \, : \! \varepsilon_2(z) \varepsilon^*_2(z) \! : - : \! \varepsilon_{\overline{2}}(z) \varepsilon^*_{\overline{2}}(z) \! : 
 - : \! \varepsilon_3(z) \varepsilon^*_3(z) \! : + : \! \varepsilon_{\overline{3}}(z) \varepsilon^*_{\overline{3}}(z) \! :$,
 \item $X(\alpha_0, z) \mapsto \, : \! \beta^*(z) \varepsilon_4(z) \! : - : \! \varepsilon^*_{\overline{4}}(z) \overline{\beta}(z) \! :
 + \omega^2 : \! \varepsilon^*_2(z) \varepsilon_{\overline{3}}(z) \! : - \omega^2 : \! \varepsilon^*_3(z) \varepsilon_{\overline{2}}(z) \! :
 + \omega : \! \varepsilon^*_1(z) \varepsilon_{\overline{4}}(z) \! : - \omega : \! \varepsilon^*_4(z) \varepsilon_{\overline{1}}(z) \! :$,
 \item $X(-\alpha_0, z) \mapsto \, : \! \beta(z) \varepsilon^*_4(z) \! : - : \! \varepsilon_{\overline{4}}(z) \overline{\beta}^*(z) \! :
 + \omega : \! \varepsilon_2(z) \varepsilon^*_{\overline{3}}(z) \! : - \omega : \! \varepsilon_3(z) \varepsilon^*_{\overline{2}}(z) \! :
+ \omega^{2} : \! \varepsilon_1(z) \varepsilon^*_{\overline{4}}(z) \! : - \omega^{2} : \! \varepsilon_4(z) \varepsilon^*_{\overline{1}}(z) \! :$,
 \item $X(\alpha_1, z) \mapsto \, : \! \varepsilon_1(z) \varepsilon^*_2(z) \! : - : \! \varepsilon_{\overline{2}}(z) \varepsilon^*_{\overline{1}}(z) \! :
 + : \! \varepsilon_3(z) \varepsilon^*_4(z) \! : - : \! \varepsilon_{\overline{4}}(z) \varepsilon^*_{\overline{3}}(z) \! :
 + : \! \varepsilon_3(z) \varepsilon^*_{\overline{4}}(z) \! : - : \! \varepsilon_4(z) \varepsilon^*_{\overline{3}}(z) \! :$,
 \item $X(-\alpha_1, z) \mapsto \, : \! \varepsilon^*_1(z) \varepsilon_2(z) \! : - : \! \varepsilon^*_{\overline{2}}(z) \varepsilon_{\overline{1}}(z) \! :
 + : \! \varepsilon^*_3(z) \varepsilon_4(z) \! : - : \! \varepsilon^*_{\overline{4}}(z) \varepsilon_{\overline{3}}(z) \! :
 + : \! \varepsilon^*_3(z) \varepsilon_{\overline{4}}(z) \! : - : \! \varepsilon^*_4(z) \varepsilon_{\overline{3}}(z) \! :$,
 \item $X(\alpha_2, z) \mapsto \, : \! \varepsilon_2(z) \varepsilon^*_3(z) \! : - : \! \varepsilon_{\overline{3}}(z) \varepsilon^*_{\overline{2}}(z) \! :$,
 \item $X(-\alpha_2, z) \mapsto \, : \! \varepsilon^*_2(z) \varepsilon_3(z) \! : - : \! \varepsilon^*_{\overline{3}}(z) \varepsilon_{\overline{2}}(z) \! :$.
\end{itemize}
Then $\rho$ is a homomorphism, and hence gives a representation of the twisted 2-toroidal Lie algebra $t(\g)$. The representation is level $(-1,0)$ for $\g = \A, \At$ and level $(-2,0)$ for $\g = \D, \Df$.
\end{theorem}

\begin{proof}
We need only show that the relations (\ref{twrelationsps}) of $t(\g)$ hold. This requires a significant number of calculations which are straightforward with repeated use of Proposition \ref{bosbracket}. Hence we show only a few of the relations below which are representative of the others; the rest can be readily verified by similar computations.

As a first example, we consider relation (3) with $(i,j)=(n-1,n)$ in the case $\g = \A$. Using Proposition \ref{bosbracket} we have:
\begin{align*}
\big [ \rho & \big ( \alpha_{n-1} (z) \big ) , \rho \big ( \alpha_n (w) \big ) \big ]
 = [: \! \varepsilon_{n-1}(z) \varepsilon^*_{n-1}(z) \! : - : \! \varepsilon_n(z) \varepsilon^*_n(z) \! : 
 \\ & - : \! \varepsilon_{\overline{n-1}}(z) \varepsilon^*_{\overline{n-1}}(z) \! : + : \! \varepsilon_{\overline{n}}(z) \varepsilon^*_{\overline{n}}(z) \! :,
 : \! \varepsilon_n(w) \varepsilon^*_n(w) \! : - : \! \varepsilon_{\overline{n}}(w) \varepsilon^*_{\overline{n}}(w) \! : ] \\
& = - ( - : \! \varepsilon^*_n(z) \varepsilon_n(z) \! : + : \! \varepsilon_n(z) \varepsilon^*_n(z) \! :)\dzmw - (-1)\partial_w \dzmw \\
& \hspace{4mm} - ( - : \! \varepsilon^*_{\overline{n}}(z) \varepsilon_{\overline{n}}(z) \! : + : \! \varepsilon_{\overline{n}}(z) \varepsilon^*_{\overline{n}}(z) \! :)\dzmw - (-1)\partial_w \dzmw \\
& = 2 \partial_w \dzmw = a_{n-1,n} \partial_w \dzmw \rho(\not{c})
 = \rho \big( [\alpha_{n-1} (z), \alpha_n (w)] \big).
\end{align*}

Secondly, consider relation (4) with $1 \leq i,j \leq n-1$; the calculation is the same in types $\A, \At$, and $\D$. In each of those cases we have
\begin{align*}
\big [ & \rho \big ( \alpha_i(z) \big ) , \rho \big ( X(\alpha_j, w) \big ) \big ]
 = [: \! \varepsilon_i(z) \varepsilon^*_i(z) \! : - : \! \varepsilon_{i+1}(z) \varepsilon^*_{i+1}(z) \! : - : \! \varepsilon_{\overline{i}}(z) \varepsilon^*_{\overline{i}}(z) \! : \\ &+ : \! \varepsilon_{\overline{i+1}}(z) \varepsilon^*_{\overline{i+1}}(z) \! : ,
 : \! \varepsilon_j(w) \varepsilon^*_{j+1}(w) \! : + : \! \varepsilon_{\overline{j+1}}(w) \varepsilon^*_{\overline{j}}(w) \! : ] \\
& = \left(-\delta_{i,j+1} : \! \varepsilon^*_i(z) \varepsilon_j(z) \! : + \delta_{ij} : \! \varepsilon_i(z) \varepsilon^*_{j+1}(z) \! : \right) \dzmw \\
& \hspace{4mm} +\left(- (-\delta_{ij}) : \! \varepsilon^*_{i+1}(z) \varepsilon_j(z) \! :  - \delta_{i+1,j} : \! \varepsilon_{i+1}(z) \varepsilon^*_{j+1}(z) \! : \right) \dzmw \\
& \hspace{4mm} +\left( -(-\delta_{ij}) : \! \varepsilon^*_{\overline{i}}(z) \varepsilon_{\overline{j+1}}(z) \! : - \delta_{i,j+1} : \! \varepsilon_{\overline{i}}(z) \varepsilon^*_{\overline{j}}(z) \! : \right) \dzmw \\
& \hspace{4mm} +\left( - \delta_{i+1,j} : \! \varepsilon^*_{\overline{i+1}}(z) \varepsilon_{\overline{j+1}}(z) \! : + \delta_{ij} : \! \varepsilon_{\overline{i+1}}(z) \varepsilon^*_{\overline{j}}(z) \! : \right) \dzmw \\
& = (2\delta_{ij} - \delta_{i,j+1} - \delta_{i+1,j}) (: \! \varepsilon_j(z) \varepsilon^*_{j+1}(z) \! : + : \! \varepsilon_{\overline{j+1}}(z) \varepsilon^*_{\overline{j}}(z) \! :) \dzmw \\
& = a_{ij} \rho \big( X(\alpha_j, z) \big) \dzmw 
 = \rho \big( [\alpha_i(z), X( \alpha_j, w)] \big).
\end{align*}

In types $\At, \D$ with $i=n$, Relation (5) is shown below.
\begin{align*}
\big [ \rho \big ( X(\alpha_n, z) \big ) , \rho \big ( X(\alpha_n, w) \big ) \big ]
& = [ \sqrt{2} (: \! \varepsilon_n(z) \varepsilon^*_{\overline{n+1}}(z) \! : - : \! \varepsilon_{\overline{n+1}}(z) \varepsilon^*_{\overline{n}}(z) \! :), \\
& \sqrt{2} (: \! \varepsilon_n(w) \varepsilon^*_{\overline{n+1}}(w) \! : - : \! \varepsilon_{\overline{n+1}}(w) \varepsilon^*_{\overline{n}}(w) \! :) ] \\
 = -2 : \! \varepsilon_n(z) \varepsilon^*_{\overline{n}}(z) \! : & \dzmw + 2 : \! \varepsilon^*_{\overline{n}}(z) \varepsilon_n(z) \! : \dzmw
 = 0.
\end{align*}

We show two examples of Relation (6). Consider type $\Df$ when $i=0, j=2$; this result is not obviously 0 a priori (cf. the same calculation is obviously 0 a priori in type $\D$, for example). But indeed,
\begin{align*}
\big [ \rho & \big ( X(\alpha_0, z) \big ) , \rho \big ( X(-\alpha_2, w) \big ) \big ] 
 = [ : \! \beta^*(z) \varepsilon_4(z) \! : - : \! \varepsilon^*_{\overline{4}}(z) \overline{\beta}(z) \! :
 + \omega^2 : \! \varepsilon^*_2(z) \varepsilon_{\overline{3}}(z) \! : \\ &- \omega^2 : \! \varepsilon^*_3(z) \varepsilon_{\overline{2}}(z) \! :
 + \omega : \! \varepsilon^*_1(z) \varepsilon_{\overline{4}}(z) \! : - \omega : \! \varepsilon^*_4(z) \varepsilon_{\overline{1}}(z) \! : ,
 : \! \varepsilon^*_2(w) \varepsilon_3(w) \! : - : \! \varepsilon^*_{\overline{3}}(w) \varepsilon_{\overline{2}}(w) \! : ] \\
& = - (-1) \omega^{2} : \! \varepsilon^*_2(z) \varepsilon_{\overline{2}}(z) \! : \dzmw - \omega^{2} : \! \varepsilon_{\overline{2}}(z) \varepsilon^*_2(z) \! : \dzmw 
 = 0.
\end{align*}

Relation (6) in type $\Df$ with $i=j=1$ is computed as
\begin{align*}
& \big [ \rho \big ( X(\alpha_1, z) \big ) , \rho \big ( X(-\alpha_1, w) \big ) \big ] 
 = [ : \! \varepsilon_1(z) \varepsilon^*_2(z) \! : - : \! \varepsilon_{\overline{2}}(z) \varepsilon^*_{\overline{1}}(z) \! :
 + : \! \varepsilon_3(z) \varepsilon^*_4(z) \! : \\ & - : \! \varepsilon_{\overline{4}}(z) \varepsilon^*_{\overline{3}}(z) \! :
 + : \! \varepsilon_3(z) \varepsilon^*_{\overline{4}}(z) \! : - : \! \varepsilon_4(z) \varepsilon^*_{\overline{3}}(z) \! : ,
 : \! \varepsilon^*_1(w) \varepsilon_2(w) \! : - : \! \varepsilon^*_{\overline{2}}(w) \varepsilon_{\overline{1}}(w) \! : \\
& + : \! \varepsilon^*_3(w) \varepsilon_4(w) \! : - : \! \varepsilon^*_{\overline{4}}(w) \varepsilon_{\overline{3}}(w) \! :
 + : \! \varepsilon^*_3(w) \varepsilon_{\overline{4}}(w) \! : - : \! \varepsilon^*_4(w) \varepsilon_{\overline{3}}(w) \! : ] \\
& = ( - : \! \varepsilon^*_2(z) \varepsilon_2(z) \! : + : \! \varepsilon_1(z) \varepsilon^*_1(z) \! : ) \dzmw
 - \partial_w \dzmw \\
& \hspace{3mm} + ( - : \! \varepsilon^*_{\overline{1}}(z) \varepsilon_{\overline{1}}(z) \! : + : \! \varepsilon_{\overline{2}}(z) \varepsilon^*_{\overline{2}}(z) \! :)\dzmw
 - \partial_w \dzmw \\
& \hspace{3mm} + ( - : \! \varepsilon^*_4(z) \varepsilon_4(z) \! : + : \! \varepsilon_3(z) \varepsilon^*_3(z) \! : ) \dzmw
 - \partial_w \dzmw
 - : \! \varepsilon^*_4(z) \varepsilon_{\overline{4}}(z) \! : \dzmw \\
& \hspace{3mm} + ( - : \! \varepsilon^*_{\overline{3}}(z) \varepsilon_{\overline{3}}(z) \! : + : \! \varepsilon_{\overline{4}}(z) \varepsilon^*_{\overline{4}}(z) \! :)\dzmw
 - \partial_w \dzmw
 + : \! \varepsilon_{\overline{4}}(z) \varepsilon^*_4(z) \! : \dzmw \\
& \hspace{3mm} - : \! \varepsilon^*_{\overline{4}}(z) \varepsilon_4(z) \! : \dzmw
+ ( - : \! \varepsilon^*_{\overline{4}}(z) \varepsilon_{\overline{4}}(z) \! : + : \! \varepsilon_3(z) \varepsilon^*_3(z) \! : ) \dzmw
 - \partial_w \dzmw \\
& \hspace{3mm} + : \! \varepsilon_4(z) \varepsilon^*_{\overline{4}}(z) \! : \dzmw
+ ( - : \! \varepsilon^*_{\overline{3}}(z) \varepsilon_{\overline{3}}(z) \! : + : \! \varepsilon_4(z) \varepsilon^*_4(z) \! :)\dzmw
 - \partial_w \dzmw \\
& = \big( : \! \varepsilon_1(z) \varepsilon^*_1(z) \! : - : \! \varepsilon_{\overline{1}}(z) \varepsilon^*_{\overline{1}}(z) \! : 
 - : \! \varepsilon_2(z) \varepsilon^*_2(z) \! : + : \! \varepsilon_{\overline{2}}(z) \varepsilon^*_{\overline{2}}(z) \! : \\
& \hspace{6mm}+ 2 : \! \varepsilon_3(z) \varepsilon^*_3(z) \! :  - 2 : \! \varepsilon_{\overline{3}}(z) \varepsilon^*_{\overline{3}}(z) \! : \big) \dzmw
 - 6 \partial_w \dzmw \\
& = \alpha_1 (z) \dzmw + r \partial_w \dzmw \rho(\not{c}) 
 = \rho \big( [X(\alpha_1, z), X(-\alpha_1, w)] \big).
\end{align*}

Other calculations are similar, noting that the element $\beta$ acts as $\varepsilon_1$ on $V$ as in \cite{JMX} (see the proof of Theorem 3.2). We show one representative example of the Serre relations (7)--(10): relation (9) in type $\At$ when $i=n, j=n-1$. We have
\begin{align*}
\big [ & \rho \big ( X(\alpha_n, z_2) \big ) , \rho \big ( X(\alpha_{n-1}, z_1) \big ) \big ] 
 = [ \sqrt{2} (: \! \varepsilon_n(z_2) \varepsilon^*_{\overline{n+1}}(z_2) \! : - : \! \varepsilon_{\overline{n+1}}(z_2) \varepsilon^*_{\overline{n}}(z_2) \! :), \\
& : \! \varepsilon_{n-1}(z_1) \varepsilon^*_n(z_1) \! : + : \! \varepsilon_{\overline{n}}(z_1) \varepsilon^*_{\overline{n-1}}(z_1) \! :] \\
& = -\sqrt{2} : \! \varepsilon^*_{\overline{n+1}}(z_2) \varepsilon_{n-1}(z_2) \! : \delta (z_2 - z_1) - \sqrt{2} : \! \varepsilon_{\overline{n+1}}(z_2) \varepsilon^*_{\overline{n-1}}(z_2) \! : \delta (z_2 - z_1). \\
& \text{Applying ad} X(\alpha_n, z_3) \text{ to this expression gives us} \\
& = [ \sqrt{2} (: \! \varepsilon_n(z_3) \varepsilon^*_{\overline{n+1}}(z_3) \! : - : \! \varepsilon_{\overline{n+1}}(z_3) \varepsilon^*_{\overline{n}}(z_3) \! :), \\
& \hspace{3mm} -\sqrt{2} : \! \varepsilon^*_{\overline{n+1}}(z_2) \varepsilon_{n-1}(z_2) \! : \delta (z_2 - z_1) - \sqrt{2} : \! \varepsilon_{\overline{n+1}}(z_2) \varepsilon^*_{\overline{n-1}}(z_2) \! : \delta (z_2 - z_1) ] \\
& = -2 : \! \varepsilon_n(z_3) \varepsilon^*_{\overline{n-1}}(z_3) \! : \delta (z_2 - z_1) \delta (z_3 - z_1) \\
& \hspace{3mm} - 2 : \! \varepsilon^*_{\overline{n}}(z_3) \varepsilon_{n-1}(z_3) \! : \delta (z_2 - z_1) \delta (z_3 - z_1). \\
& \text{Applying ad} X(\alpha_n, z_4) \text{ to the latter expression gives the desired result of } 0. 
\end{align*}

\end{proof}

Finally, we remark that since the above construction of $t(\g)$ contains the associated twisted affine Lie algebra $\hat{\g}$ at the 0-mode, then the preceding theorem gives a bosonic free field realization of the twisted affine Lie algebra of type $D_4^{(3)}$ for the first time (cf. \cite{FF}).

\bibliographystyle{amsalpha}

\begin{thebibliography}{ABGP}
\bibitem[ABGP]{ABGP} B.~N.~Allison, S.~Azam, S.~Berman, Y.~Gao and A.~Pianzola,
{\em Extended affine Lie algebras and their root systems}, Memoir Amer. Math. Soc. {\bf 126} (1997), 1-122. 

\bibitem[BR1]{BR1}   P. ~Batra and S. ~E. ~Rao, {\em Classification of irreducible integrable modules for twisted toroidal Lie algebras with finite-dimensional weight spaces,} Pac. J. Math. {\bf 237} (2008), no.1, 151-181.

\bibitem[BR2]{BR2}   P. ~Batra and S. ~E. ~Rao, {\em On integrable modules for the twisted full toroidal Lie algebra,} J. Lie Theory. {\bf 28} (2018), 79-105.

\bibitem[BB]{BB} S. ~Berman and Y. ~Billig, {\em Irreducible representations for toroidal Lie algebras}, J. Algebra {\bf 221} (1999), 188-231.

\bibitem[BBS]{BBS}  S. ~Berman, Y. ~Billig and J. ~Szmigielski, {\em Vertex Operator Algebras and the Representation Theory of Toroidal Algebras}, Contemp. Math. {\bf 297} (2002), 1-26.

\bibitem[BK]{BK}   S. ~Berman and Y. ~Krylyuk, {\em Universal central extensions of twisted and untwisted Lie algebras extended over commutative rings,} J. Algebra. {\bf 173} (1995), 302-347.

\bibitem[B]{B}  Y. Billig, {\em Energy-Momentum Tensor for the Toroidal Lie Algebras,} (2002), 1-45, arXiv:math/0201313v1.

\bibitem[C]{C}  R. ~Carter, {\em Lie Algebras of Finite and Affine Type}, Cambridge Studies in Advanced Mathematics no. $96$, Cambridge University Press, Cambridge, 2005.

\bibitem[CL]{CL}   V. ~Chari and T. ~Le, {\em Representations of double affine Lie algebras,}  A tribute to C. S. Seshadri (Chennai, 2002), Trends Math. Birkhauser, Basel, (2003), 199-219.

\bibitem[D]{D} V. G. Drinfeld, {\em A new realization of Yangians and quantum affine algebras}, Soviet Doklady {\bf 36} (1987), 212-216.

\bibitem[EM]{EM} S.~Eswara Rao and R.~V.~Moody, {\em Vertex representations for $n$-toroidal Lie algebras and a generalization of the Virasoro algebras},
Comm. Math. Phys. {\bf 159} (1994), 239-264.

\bibitem[FM]{FM} M. ~Fabbri and R.~V. ~Moody, {\em Irreducible representations of Virasoro-toroidal Lie algebras}, Comm. Math. Phys. {\bf 159} (1994), 1-13.

\bibitem[FF]{FF} A.~Feingold and I.~B.~Frenkel,
{\em Classical affine algebras}, Adv. Math. {\bf 56} (1985),
117-172.

\bibitem[F1]{F1} I.~B.~Frenkel, {\em Spinor representations of affine Lie algebras}, Proc. Nat. Acad. Sci. U.S.A. {\bf 77} (1980), 6303-6306.

\bibitem[F2]{F2} I.~B.~Frenkel,  {\em Two constructions of affine Lie algebra representations and boson–fermion correspondence in quantum field theory}, J. Func. Anal. {\bf 44}, (1981). pp. 259-327.
     
\bibitem[F3]{F3} I.~B.~Frenkel,  {\em Representations of Kac-Moody algebras and dual resonance models}, 
(Chicago, 1982), Lect. Appl. Math. {\bf 21}, Amer. Math. Soc.,
Providence, 1985. pp. 325-353.

\bibitem[FJW]{FJW} I. Frenkel, N. Jing and W. Wang, {\em Vertex representations
via finite groups and the McKay correspondence}, Int. Math. Res.
Not. {\bf 4} (2000), 195-222.

\bibitem[FLM]{FLM} I.~B.~Frenkel, J.~Lepowsky and A.~Meurman,
{\em Vertex operator algebras and the Monster}, Academic Press,
New York, 1988.

\bibitem[FJ]{FJ} J.~Fu and C. Jiang, {\em Integrable representations for the twisted full toroidal Lie algebras}, J. Algebra {\bf 307} (2007), 769-794.

\bibitem[G]{G} Y. Gao, {\em Fermionic and bosonic representations of
the extended affine Lie algebra ${gl}_N({\Bbb C}_q)$}. Canad.
Math. Bull. {\bf 45} (2002), 623-633. 

\bibitem[JMM]{JMM} N.~Jing, C.~R.~Mangum and K.~C.~Misra, {\em On realization of some twisted toroidal Lie algebras}, Contemp. Math. {\bf 695} (2017), 139-148.

\bibitem[JMM2]{JMM2} N.~Jing, C.~R.~Mangum and K.~C.~Misra, {\em Fermionic realization of twisted toroidal Lie algebras}, arXiv:1812.02584v1.

\bibitem[JM]{JM} N.~Jing and K.~C.~Misra, {\em Fermionic realizations of toroidal Lie algebras of classical types}, J. Algebra {\bf 324} (2010), 183 -194.

\bibitem[JMT]{JMT} N.~Jing, K.~C.~Misra and S.~Tan, {\em Bosonic realizations of higher level toroidal Lie algebras}, Pacific J. Math. {\bf 219} (2005), 285-302.

\bibitem[JMX]{JMX} N.~Jing, K.~C.~Misra and C.~Xu, {\em Bosonic realization of toroidal Lie algebras of classical types}, Proc. of the Amer. Math. Soc. {\bf 137} (2009), 3609-3618.

\bibitem[JX]{JX} N.~Jing and C.~Xu, {\em Toroidal Lie superalgebras and free field representations}, Contemp. Math. 623 (2014), 135-153.

\bibitem[K1]{K1} V.~G.~Kac, {\em Infinite dimensional Lie algebras}, 3rd. Ed.,
Cambridge University Press, Cambridge, 1990.

\bibitem[K2]{K2} V.~G.~Kac, {\em Vertex algebras for beginners}, 2nd. Ed., University Lecture Series vol.$10$, American Mathematical Society, 1998.

\bibitem[KP]{KP} V.~G. Kac, D. Peterson, {\em Spin and wedge representations of infinite-dimensional Lie algebras and groups}, Proc. Natl. Acad. Sci. USA 78 (1981), 3308-3312.

\bibitem[KL1]{KL1} F.~T. ~Kroode and J.~van de Leur,
{\em Bosonic and Fermionic Realizations of the Affine Algebra $\hat{gl}_n$}, Commun. Math. Phys. {\bf 137} (1991),
67-107.

\bibitem[KL2]{KL2} F.~T. ~Kroode and J.~van de Leur,
{\em Bosonic and Fermionic Realizations of the Affine Algebra $\hat{so}_{2n}$}, Commun. in Alg. {\bf 20} (1992), no.11, 3119-3162.

\bibitem[L]{L} M. Lau, {\em Bosonic and fermionic representations of Lie
algebra central extensions}. Adv. Math. {\bf 194} (2005),
225-245.

\bibitem[MRY]{MRY} R.~V.~Moody, S.~E.~Rao and T.~Yokonuma,
{\em Toroidal Lie algebras and vertex
representations}, Geom. Dedicata {\bf 35} (1990),
283-307.

\bibitem[NSW]{NSW} E. ~Neher, A. ~Savage and W.~Wang, {\em Geometric representation theory and extended affine Lie algebras}, American Mathematical Society, Providence, RI, 2011.

\bibitem[R]{R} S.~Eswara Rao, {\em On Representations of Toroidal Lie Algebras},
for Proceedings of Functional Analysis VIII, (2005), 1-29, arXiv:math/0503629v1.

\bibitem[T1]{T1} S. Tan, {\em Principal construction of the toroidal Lie algebra of type $A_1$}, Math. Zeit. {\bf 230} (1999), 621-657.

\bibitem[T2]{T2} S. Tan, {\em Vertex operator representations for toroidal Lie
algebras of type $B_l$}, Comm. Algebra {\bf 27} (1999), 3593-3618.

\bibitem[V]{V} J. ~van de Leur, {\em Twisted Toroidal Lie Algebras}, (2008), 1-16, arXiv:math/0106119v1.


\end{thebibliography}

\end{document}